\documentclass{lms}

\usepackage{tikz}
\usepackage{amssymb}
\usepackage{amsmath}
\usepackage{hyperref}

\usepackage{tikz-3dplot}
\tdplotsetmaincoords{70}{110}

\usepackage{subcaption}

\usetikzlibrary{shapes.geometric}
\usepackage{mathrsfs}
\usetikzlibrary{arrows}
\usetikzlibrary{shapes.geometric}

\newtheorem{theorem}{Theorem}[section] 
\newtheorem{lemma}[theorem]{Lemma}     
\newtheorem{corollary}[theorem]{Corollary}
\newtheorem{proposition}[theorem]{Proposition}

\newnumbered{assertion}{Assertion}    
\newnumbered{conjecture}{Conjecture}  
\newnumbered{definition}{Definition}
\newnumbered{hypothesis}{Hypothesis}
\newnumbered{remark}{Remark}
\newnumbered{note}{Note}
\newnumbered{observation}{Observation}
\newnumbered{problem}{Problem}
\newnumbered{question}{Question}
\newnumbered{algorithm}{Algorithm}
\newnumbered{example}{Example}
\newunnumbered{notation}{Notation} 

\simpleequations

\providecommand{\Ric}{\mathop{\rm Ric}\nolimits}
\providecommand{\bc}{\mathop{\rm bc}\nolimits}
\providecommand{\Aut}{\mathop{\rm Aut}\nolimits}
\providecommand{\DF}{\mathop{\rm DF}\nolimits}
\providecommand{\ZZ}{\mathbb{Z}}
\providecommand{\QQ}{\mathbb{Q}}
\providecommand{\RR}{\mathbb{R}}
\providecommand{\CC}{\mathbb{C}}
\providecommand{\PP}{\mathbb{P}}
\providecommand{\X}{\mathcal{X}}
\providecommand{\inte}{\text{int}}
\providecommand{\Hom}{\text{Hom}}
\renewcommand{\L}{\mathcal{L}}
\providecommand{\wdiv}{\mathop{\rm Div}\nolimits}
\providecommand{\vol}{\mathop{\rm vol}\nolimits}

\def\diagram1{

\definecolor{wrwrwr}{rgb}{0.3803921568627451,0.3803921568627451,0.3803921568627451}
\definecolor{rvwvcq}{rgb}{0.08235294117647059,0.396078431372549,0.7529411764705882}

\begin{tikzpicture}[scale=0.4,line cap=round,line join=round,>=triangle 45,x=1cm,y=1cm]
\clip(-8.69388785352341,-6.85149354395804) rectangle (12.217649006953684,5.977670174126094);
\fill[line width=0.4pt,color=rvwvcq,fill=rvwvcq,fill opacity=0.10000000149011612] (3.95,-3.22) -- (-4.499755258769425,-0.465188547516475) -- (-4.021306091512521,1.7101502921524365) -- (3.8019942792933907,3.9467619721580767) -- (9.206358516755902,2.6979556386504635) -- (7.16248304580607,-1.3101637913940103) -- cycle;
\draw [line width=0.4pt,color=wrwrwr] (-4.499755258769425,-0.465188547516475)-- (-4.021306091512521,1.7101502921524365);
\draw [line width=0.4pt,color=wrwrwr] (-4.021306091512521,1.7101502921524365)-- (3.8019942792933907,3.9467619721580767);
\draw [line width=0.4pt,color=wrwrwr] (-4.499755258769425,-0.465188547516475)-- (3.95,-3.22);
\draw [line width=0.4pt,color=wrwrwr] (2.4376800090649002,0.7867993541259463)-- (-16.628346540838827,-10.065412601807587);
\draw [line width=0.4pt,color=rvwvcq] (3.95,-3.22)-- (-4.499755258769425,-0.465188547516475);
\draw [line width=0.4pt,color=rvwvcq] (-4.499755258769425,-0.465188547516475)-- (-4.021306091512521,1.7101502921524365);
\draw [line width=0.4pt,color=rvwvcq] (-4.021306091512521,1.7101502921524365)-- (3.8019942792933907,3.9467619721580767);
\draw [line width=0.4pt,color=rvwvcq] (3.8019942792933907,3.9467619721580767)-- (9.206358516755902,2.6979556386504635);
\draw [line width=0.4pt,color=rvwvcq] (9.206358516755902,2.6979556386504635)-- (7.16248304580607,-1.3101637913940103);
\draw [line width=0.4pt,color=rvwvcq] (7.16248304580607,-1.3101637913940103)-- (3.95,-3.22);
\draw [line width=0.4pt,color=wrwrwr] (-4.499755258769425,-0.465188547516475)-- (-6.0703572252974505,-0.590007254352775);
\draw [line width=0.4pt,color=wrwrwr,domain=-8.69388785352341:12.217649006953684] plot(\x,{(-7.125388691254146-1.5706019665280255*\x)/0.12481870683630003});
\draw [line width=0.4pt,color=wrwrwr] (-0.2748776293847124,-1.8425942737582377)-- (-0.6233328024886629,-2.9114009714819993);
\draw (2.430829827729521,1.0842605845140028) node[anchor=north west] {$b$};
\draw (-4.2953174358945745,-2.746161154171117) node[anchor=north west] {$p_w$};
\draw (-2,-1.4) node[anchor=north west] {$q$};
\draw (-0.65,-2.471250503069314) node[anchor=north west] {$n$};
\draw (-8.3,4) node[anchor=north west] {$\Pi(w,a)$};
\draw [->,line width=0.05pt,color=wrwrwr] (-0.2748776293847124,-1.8425942737582377) -- (-0.6233328024886629,-2.9114009714819993);
\draw [->,line width=0.05pt,color=wrwrwr] (-4.499755258769425,-0.465188547516475) -- (-6.0703572252974505,-0.590007254352775);
\draw (-6.3296562540479115,-0.505) node[anchor=north west] {$w$};
\draw (-0.20831242284778076,2.0372841750002526) node[anchor=north west] {$\Huge{P}$};
\draw (-0.4648956972094629,-0.5652033220968145) node[anchor=north west] {$O$};
\draw (-4.533573333516137,0.2) node[anchor=north west] {$a$};
\begin{scriptsize}
\draw [fill=wrwrwr] (-4.499755258769425,-0.465188547516475) circle (1.5pt);
\draw [fill=wrwrwr] (-0.4676952405849731,-0.8669142782450808) circle (1.5pt);
\draw [fill=wrwrwr] (2.4376800090649002,0.7867993541259463) circle (1.5pt);
\draw [fill=wrwrwr] (-4.294715291030033,-3.0452198962684096) circle (2pt);
\draw [fill=wrwrwr] (-1.4873591410315417,-1.4472978596114847) circle (1.5pt);
\end{scriptsize}
\end{tikzpicture}
}

\title[$R(X)$ for Fano $T$-manifolds of Complexity $1$]{Greatest Lower Bounds on Ricci Curvature for Fano $T$-manifolds of Complexity $1$} 

\author{Jacob Cable}

\classno{32Q20 (primary), 53C55, 14L30, 14J45 (secondary)}

\begin{document}
\maketitle

\begin{abstract}
In this short note we determine the greatest lower bounds on Ricci curvature for all Fano \(T\)-manifolds of complexity one, generalizing the result of Chi Li. Our method of proof is based on the work of Datar and Sz\'ekelyhidi, using the description of complexity one special test configurations given by Ilten and  S{\"u}{\ss}.
\end{abstract}

\section{Introduction} 
\label{intro}
A fundamental problem in K\"ahler geometry is finding K\"ahler-Einstein metrics. The cases of \(c_1(X)<0\) and \(c_1(X) = 0\) have been settled by Aubin \cite{Aubin1978} and Yau \cite{Yau1977},  and so all that remains is the case of Fano manifolds, where \(-K_X\) is ample. Fix a K\"ahler form \(\omega \in 2 \pi c_1(X)\). The problem is then to find a K\"ahler form \(\omega' \) in the same class such that 
\[
\Ric(\omega') = \omega'.
\]
A powerful technique for dealing with this equation is the continuity method, where for \(t \in [0,1]\) one considers solutions \(\omega_t\) to the continuity path:
\[
\Ric(\omega_t) = t\omega_t + (1-t) \omega.
\]
From \cite{Yau1977} we always have a solution for \(t = 0\). However Tian \cite{Tian1992} showed that for some \(t\) sufficiently close to \(1\) there may not be a solution for certain Fano manifolds. In this case it is natural to ask for the supremum of permissible \(t\), which turns out to be independent of the choice of \(\omega\). This invariant was first discussed, although not explicitly defined, by Tian in \cite{Tian1987}. It was first explicitely defined by Rubenstein in \cite{Rubinstein2008} and was further studied by Szekelyhidi in \cite{Szekelyhidi2011}. Tian's invariant is defined as follows:
\begin{definition}
Let \((X,\omega)\) be a K\"ahler manifold with \(\omega \in 2 \pi c_1(X)\). Define:
\[
R(X) := \sup ( t \in [0,1]   : \exists \ \omega_t \in 2 \pi c_1(X)  \ \Ric( \omega_t) = t \omega_t + (1-t) \omega ).
\]
\end{definition}
In \cite{Rubinstein2008} Rubenstein showed relation between \(R(X)\) and Tian's alpha invariant \(\alpha(X)\), and in \cite{Rubinstein2009} conjectured that \(R(X)\) characterizes the \(K\)-semistability of \(X\). This conjecture was later verified by Li in \cite{Li2017}.

In \cite{Li2011} Li  determined a simple formula for \(R(X_\Delta)\), where \(X_\Delta\) is the polarized toric Fano manifold determined by a reflexive lattice polytope \(\Delta\). This result was later recovered in \cite{Datar2016}, by Datar and Sz\'ekelyhidi, using notions of \(G\)-equivariant \(K\)-stability. Using this same method we obtain an effective formula for manifolds with a torus action of complexity one, in terms of the combinatorial data of its divisorial polytope. Previously \(R(X)\) has been calculated for group compactifications by Delcroix \cite{Delcroix2017} and for homogeneous toric bundles by Yao \cite{Yao2017}.

Let \(X\) be a complexity one Fano \(T\)-variety, that is a Fano variety admitting an effective action of an algebraic torus with maximal orbits of codimension one. We have mutually dual character and cocharacter lattices \(M = \Hom(T,\CC^*), \ N = \Hom(\CC^*,T)\) respectively. Similar to toric varieties there exists a combinatorial description for \(X\). Let \(\mu: X \to M_\mathbb{R} \) be the moment map for the torus action. Its image is a lattice polytope \(\Box\). The push-forward of the Liouville measure under \(\mu\) is known as the Duistermaat-Heckman measure on \(\Box\), which we denote \(\nu\).

By the criterion for \(K\)-stability in \cite{Ilten2017} we know that if the weighted barycenter \(\bc_\nu(\Box)\) coincides with the origin then \(X\) is \(K\)-stable and so \(R(X) = 1\). Moreover it is shown that there are finitely many normal equivariant toric degenerations of \(X\) and that the corresponding lattice polytopes \(\Delta_1,\dots,\Delta_m \subseteq M'_\RR := M_\RR \times \RR\) may be described using the combinatorial data for \(X\).

Suppose now \(\bc_\nu(\Box) \neq 0\). Let \(q\) be the intersection of the ray generated by \(-\bc_\nu(\Box)\) with \(\partial \Box\). Consider the halfspace \(H := N_\RR \times \RR^+ \subset N'_\RR\). Let \(q_i\) be the point of intersection of \(\partial \Delta_i\) with the ray generated by \(-\bc(\Delta_i)\), where \(\bc(\Delta_i)\) is the barycenter of \(\Delta_i\). This is well defined since \(\pi(\bc(\Delta_i)) = \bc_\nu(\Box)\), where \(\pi\) is the projection to \(M_\RR\). Let \(F_i\) be  the face of \(\Delta_i\) in which \(q_i\) lies, and let \(S\) be the set of indices \(i\) for which all outer normals to \(F_i\) lie in \(H\). We may now state our result:
\begin{theorem} Let \(X\) be a complexity one Fano \(T\)-variety as above. If \(\bc_\nu(\Box) = 0\) then \(R(X) = 1\). Otherwise we have:
\[
R(X) = \min \  \left\lbrace \frac{|q|}{|q-\bc_\nu(\Box)|}  \ \right\rbrace \cup \ \left\lbrace \frac{|q_i|}{|q_i - \bc(\Delta_i)|} \right\rbrace_{i \in S} .
\]
\ \\
\end{theorem}
\begin{example}
Consider the \((\CC^*)^2\)-threefold 2.30 from the list in \cite{Mori1981}. From \cite{Sues2014} we see there are \(4\) normal toric degenerations, given by the polytopes \(\Delta_1,\dots,\Delta_3\). It can be checked in this case that \(S = \emptyset\), as for each \(i\) there is an outer normal \(n_i \not\in H\) to the face \(F_i\). See Figure 1 (a) for an example.
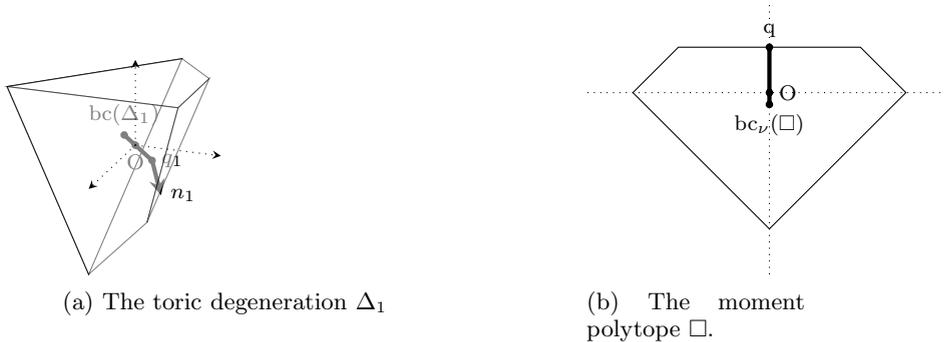
\begin{figure}[h]
\centering
\begin{subfigure}[t]{0.4\textwidth}
\begin{tikzpicture}[%
    tdplot_main_coords,
    scale=0.6,
    >=stealth
  ]
  \draw[->, ultra thick, opacity = 1] (0, 18/46, -57/184)  -- (0, 0.6, -1.1) node[right]{$n_1$} ;
    \draw[fill=white,opacity=1] (0,-3,1) -- (-3,0,1) -- (3,0,-2) -- cycle;
    \draw[fill=white,opacity=0] (3,0,-2) -- (-3,0,1) -- (-2,1,1) -- (2,1,-1) -- cycle;
    \filldraw[ultra thick, opacity = 1] (0, -6/23, 19/92) circle (1pt) node[above]{$\bc(\Delta_1)$}  -- (0,0,0) circle (1pt) node[below]{O} -- (0, 18/46, -57/184) circle (1pt) node[right]{$q_1$}	 ;
    \draw[fill=white,opacity=0.5] (0,-3,1) -- (-3,0,1) -- (-2,1,1) -- (0,1,1) -- cycle;
    \draw[fill=white,opacity=0.5] (0,-3,1) -- (0,1,1) -- (2,1,-1) -- (3,0,-2) -- cycle;
    \draw[fill=white,opacity=0.5] (0,1,1) -- (-2,1,1) -- (2,1,-1) -- cycle;
     \draw[->,dotted] (0,0,0) -- (0,0,2);
    \draw[->,dotted] (0,0,0) -- (3,0,0);
    \draw[->,dotted] (0,0,0) -- (0,2,0);
\end{tikzpicture}
\caption{The toric degeneration \(\Delta_1\)}
\end{subfigure}
\qquad \qquad \qquad
\begin{subfigure}[t]{0.2\textwidth}
\begin{tikzpicture}[scale = 0.6]
\draw (-2,1) -- (-3,0) -- (-3,0) -- (0,-3) -- (3,0) -- (2,1) -- cycle ;
\filldraw[ultra thick, opacity = 1] (0,-6/23) circle (1pt) node[below]{$\bc_\nu(\Box)$}  -- (0,0) circle (1pt) node[right]{O} -- (0, 1) circle (1pt) node[above]{q}	 ; 
    \draw[dotted] (-4,0) -- (4,0);
    \draw[dotted] (0,-4) -- (0,2);
\end{tikzpicture}
\caption{The moment polytope \(\Box\).}
\end{subfigure}
\caption{Determining \(R(X)\) for threefold 2.30.}
\end{figure}
Therefore \(R(X)\) is given by the first term in the minimum. We calculate \(\bc_\nu(\Box)= (0,-6/23)\) and \(q = (0,1)\). Then:
\[
R(X) = \frac{1}{1+ 6/23} = \frac{23}{29}.
\]
\end{example}
\begin{corollary}
In the table below we calculate \(R(X)\) for \(X\) a Fano threefold admitting a \(2\)-torus action\footnote{We have indicated with * when we refer only to a particular element of the deformation family which admits a \(2\)-torus action.} appearing in the list of Mori and Mukai \cite{Mori1981}. We only include those where \(R(X) <1\). Note all admit a K{\"a}hler-Ricci soliton by \cite{Ilten2017} and \cite{Cable2018}.
\begin{table}[h]
\begin{tabular}[h]{|c|c|}
\hline
X & R(X) \\
\hline
2.30 & \(23/29\) \\
2.31 & \(23/27\) \\
3.18 & \(48/55\) \\
3.21 & \(76/97\) \\
3.22 & \(40/49\) \\
3.23 & \(168/221\) \\
3.24 & \(21/25\) \\
4.5* & \(64/69\) \\
4.8 & \(76/89\) \\
\hline
\end{tabular}
\caption{Calculations for complexity \(1\) threefolds appearing in the list of Mori and Mukai for which \(R(X) <1\).}
\label{threefoldtable}
\end{table}
\end{corollary} 
\begin{remark}
It is worth noting that for each threefold in Table 1 we have \(S = \emptyset\), and so \(R(X)\) may be calculated from the data \((\Box,\bc_\nu(\Box))\).
\end{remark}
\section{G-equivariant K-stability}
Let \(X\) be a Fano manifold with the action of a complex reductive group \(G\) of automorphisms containing a maximal torus \(T\). Fix a \(T\)-invariant K\"ahler form \(\omega \in 2 \pi c_1(X)\) induced by the Fano condition. Recall that the Lie algebra \(\mathfrak{t}\) of the maximal compact torus in \(T\) may be identified with \(N_\RR = N \otimes \RR\).
\begin{definition}
A \(G\)-equivariant test configuration for \((X,G)\) is a \(\CC^*\)-equivariant flat family \(\X\) over the affine line equipped with a relatively ample equivariant \(\QQ\)-line bundle \(\mathcal{L}\) such that 
\begin{enumerate}
\item The \(\CC^*\)-action \(\lambda\) on \((\X, \mathcal{L})\) lifts the standard action on \(\mathbb{A}^1\);
\item The general fiber is isomorphic to \(X\) and \(\mathcal{L}\) is the relative anti-canonical bundle of \(\X \to \mathbb{A}^1\).
\item The action of \(G\) extends to \((\X,\mathcal{L})\) and commutes with the \(\CC^*\)-action \(\lambda\).
\end{enumerate}
A test configuration with \(\X \cong X \times \mathbb{A}^1\) is called a product configuration. If such an isomorphism exists and is \(\CC^*\)-equivariant then we call the test configuration trivial. Finally a test configuration with normal special fiber is called special.
\end{definition}
In this note we work with \(G = T\) being a maximal torus in \(\Aut(X)\). We then have an induced \(T' = T \times \CC^*\)-action on the special fiber. We follow the conventions in \cite{Berman2014} for moment maps and Hamiltonian functions. There is a canonical lift of \(T'\)-action to \(-K_{\X_0}\), inducing a canonical choice of moment map \(\mu: \X_0 \to M_\RR'\). The restriction of \(\lambda\) to \(\X_0\) is generated by the imaginary part of a \(T'\)-invariant vector field \(w\), and by an abuse of notation we also write \(w \in N'_\RR\) for the corresponding  one-parameter subgroup. The moment map \(\mu\) then specifies a Hamiltonian function \(\theta_w := \langle \mu, w \rangle: \X_0 \to \RR  \).
\begin{definition}
The twisted Donaldson-Futaki character of a special test configuration \((\X, \mathcal{L}) \) is given by:
\[
\DF_t(\X,\mathcal{L})(w) = \DF(\X,\mathcal{L})(w) + \frac{(1-t)}{V} \int_{\X_0} ( \max_{\X_0} \theta_w - \theta_w) \ \omega^n . 
\]
where \(V = \frac{1}{n!} \int_{\X_0} \omega^n\) is the volume of \(\X_0\), and \(\DF(\X,\L)(w) = \frac{1}{V} \int_{\X_0} \theta_w \omega^n\) is the classical Donaldson-Futaki invariant of the configuration, in the form given in \cite[Lemma 3.4]{Berman2014}.
\end{definition}
We say the pair \((X,t)\) is \(G\)-equivariantly \(K\)-semistable if \( \DF_t(\X,\mathcal{L}) \ge 0\) for all \(G\)-equivariant special configurations \((\X,\mathcal{L})\). We then have:
\begin{theorem}[{\cite[Proposition 10]{Datar2016}} ]
Let \(X\) be a polarized Fano manifold, with K\"ahler form \(\omega\). Let \(t \in [0,1]\). Then \((X,t)\) is \(G\)-equivariantly \(K\)-semistable only if for all \(s <t\)  there exists \(\omega_s \in 2 \pi c_1(X)\) such that \(\Ric(\omega_s) = s \omega_s + (1-s) \omega\).
\end{theorem}
\begin{remark}
It follows from this that:
\begin{align*}
R(X) = \inf_{(\X,\L)}( \sup(t | \DF_t(\X,\L) \ge 0) ),
\end{align*}
where \((\X,\L)\) varies over all special test configurations for \((X,L)\).
\end{remark}
\section{A short digression into convex geometry}
Our result relies on an observation regarding certain families of piecewise affine functions on convex polytopes. Let \(V\) be a real vector space and \(P \subset V\) be a convex polytope containing the origin, with \(\dim P = \dim V\). Fix some point \(b \in \text{int}(P)\). Let \(q \in \partial P\) be the intersection of \(\partial P\) with the ray \(\tau = \RR^+ (-b)\). Suppose \(n \in V^\vee\) is an outer normal to a face containing \(q\).

For \(a \in \partial P\) write \(\mathcal{N}(a) = \{w\in V^\vee \ | \langle a,w \rangle = \max_{x \in P} \langle x,c \rangle \}\). For \(w \in \mathcal{N}(a) \) let \(\Pi(a,w)\) be the affine hyperplane tangent to \(P\) at \(a\) with normal \(w\). For \(w \in \inte(\tau^\vee) \) there is a well-defined point of intersection of \(\Pi(a,w)\) and \(\tau\) which we denote \(p_w\). See Figure 2 for a schematic.
\begin{figure}[h]
\diagram1
\caption{An Example in \(V \cong \RR^2\)}
\label{schematic}
\end{figure}
\begin{lemma}
Fix \(w \in \inte(\tau^\vee) \backslash (\RR^+ n)\). For \(s \in [0,1]\) set \(w(s) := sn + (1-s)w\). As \(n \in \tau^\vee \) we may consider \(p(s) := p_{w(s)}\). For \(0 \le s' < s \le 1\) we then have:
\[
\frac{|p(s)|}{|p(s)-b|} < \frac{|p(s')|}{|p(s')-b|}.
\]
\end{lemma}
\begin{proof}
Without loss of generality we may assume \(s' = 0\). For \(s \in [0,1]\) the points \(p(s), q,b\) are collinear, so \(|p(s)|= |p(s)-q|+|q|\) and \(|p(s)-b| =|p(s)-q| +|q| + |b| \). Therefore:
\[
\frac{|p(s)|}{|p(s) - b|} = \frac{|p(s)-q|+|q|}{|p(s)-q| +|q| + |b|}.
\]
Hence it is enough for \(|p(s)-q| < |p(0)-q|\) whenever \(s >0\). Since \(q \neq 0\) is fixed this is equivalent to:
\[
\frac{|p(s) - q|}{|q|} < \frac{|p(0) - q|}{|q|}.
\]
For each \(s \in [0,1]\) choose \(a(s) \in \partial P\) such that \(w(s) \in \mathcal{N}(a(s))\). Write \(a = a(0)\) for convenience. We then have:
\[
\frac{|p(s)-q|}{|q|} = \frac{\langle a(s)-q, w \rangle }{\langle q,w \rangle}.
\]
Note \(n \in \mathcal{N}(q)\). Now \(\langle a(s)-q,n \rangle \le 0\)  and \( \langle a(s) - q,w \rangle \le \langle a - q,w \rangle\). Clearly we have \(\langle q , n \rangle > 0\). Then:
\begin{align*}
\frac{\langle a(s) - q, w(s) \rangle }{\langle q, w(s) \rangle} &= \frac{s\langle a(s)-q, n \rangle + (1-s)\langle a(s)-q,w \rangle}{s \langle q , n \rangle + (1-s) \langle q, w \rangle} \\ &\le \frac{(1-s)\langle a-q, w \rangle }{s \langle q , n \rangle + (1-s) \langle q, w \rangle} \\ &< \frac{\langle a-q, w \rangle }{\langle q,w \rangle}.
\end{align*}
\end{proof}
\begin{corollary}
Let \(V,P,b,q,\tau,n\) be as in the introduction to this section. Fix some open halfspace \(H \subset V^\vee\) given by \(u \ge 0\) for some \(u \in V \backslash \{0\}\). This defines a projection map \(\pi: V \to V/\langle u \rangle.\) Consider the function \(F_b: V^\vee \times  [0,1] \to \RR\) given by:
\[
F_b(w,t) := t \langle b,w \rangle+ (1-t) \max_{x \in P} \langle x, w \rangle
\]
For any \(W \subseteq V^\vee\) containing \(n\) we have:
\begin{equation}
\sup (t \in [0,1] \ | \ \forall_{w \in W} \  F_b(t,w) \ge 0) = \frac{|q|}{|q-b|}.
\end{equation}
If for some choice of \(n\) we have \(n \not\in H\) then:
\begin{equation}
\sup (t \in [0,1] \ | \ \forall_{w \in H} \  F_b(t,w) \ge 0) = \frac{|\tilde{q}|}{|\tilde{q} - \pi(b)|},
\end{equation}
where \(\tilde{q}\) is the intersection of the ray \(\pi(\tau)\) with the boundary of \(\pi(P)\).
\end{corollary}
\begin{proof}
Note that:
\[
\sup (t \in [0,1] \ | \ \forall_{w \in W} \  F_b(t,w) \ge 0) = \inf_{w \in W} \sup (t \in [0,1] \ | \  F_b(t,w) \ge 0).
\]
Moreover \(\sup (t \in [0,1] \ | \  F_b(t,w) \ge 0) = 1 > F_b(t,n)\) for \(\langle b,w \rangle \ge 0\), so without loss of generality we may assume \(W \subseteq \inte(\tau^\vee)\). For \(w \in W\) then:
\begin{align*}
\sup (t \in [0,1] \ | \  F_b(t,w) \ge 0) &= \frac{\max_{x \in P} \langle x, w \rangle}{ \max_{x \in P} \langle x, w \rangle - \langle b, w \rangle } \\ &= \frac{ \langle a,w \rangle}{\langle a ,w \rangle - \langle b,w \rangle} \\ &= \frac{ \langle p_w,w \rangle}{\langle p_w ,w \rangle - \langle b,w \rangle} = \frac{|p_w|}{|p_w-b|}.
\end{align*}
Hence:
\[
\sup (t \in [0,1] \ | \ \forall_{w \in W} \  F_b(t,w) \ge 0) = \inf_{w \in W} \frac{|p_w|}{|p_w-b|}.
\]
Now for \(w \in W\) consider the continuity path \(w(s) = sn + (1-s)w\). By Lemma 1 if \(n \in W\) then the above infimum is attained when \(s=1\) and we obtain (1). Otherwise the infimum is attained at some \(w \in \partial W\). For (2) restricting \(F_b\) to \(\partial H \times [0,1]\) gives:
\begin{align*}
F_b(w,t) =  t \langle \pi(b) ,w \rangle+ (1-t) \max_{x \in \pi(P)} \langle x, w \rangle.
\end{align*}
Applying (1) to the polytope \(\pi(P)\) in the vector space \(\partial H\) we obtain (2).
\end{proof}
\section{Proof of Theorem 1.1}
Let \(T\) be an algebraic torus. We recall the definition of a divisorial polytope, which is used to generalize the polytope description of a polarized toric variety to the complexity one case. We will be able to calculate \(R(X)\) for our varieties using this combinatorial data.
\begin{definition}
A divisorial polytope function \(\Psi\) on a lattice polytope \(\Box \subset M_\RR\):
\[
\Psi: \Box \to \wdiv_\RR \PP^1, \ u \mapsto \Sigma_{y \in \PP^1} \Psi_y(u) \cdot \{y\},
\]
such that:
\begin{itemize}
\item For \(y \in \PP^1\) the function \(\Psi_y: \Box \to \RR\) is the minimum of finitely many  affine functions, and \(\Psi_y \equiv 0\) for all but finitely many \(y \in \PP^1\).
\item Each \(\Psi_y\) takes integral values at the vertices of the polyhedral decomposition its regions of affine linearity induce on \(\Box\).
\item \(\deg \Psi(u) > -2\) for \(u \in \text{int} (\Box)\);
\end{itemize}
A divisorial polytope is said to be Fano if additionally we have that:
\begin{itemize}
\item The origin is an interior lattice point of \(\Box\).
\item The affine linear pieces of each  \(\Psi_y\) are of the form \(u \mapsto \frac{\langle v,u \rangle - \beta + 1}{\beta}\) for some primitive lattice element \(v \in N\);
\item Every facet \(F\) of \(\Box\) with \((\deg \circ \Psi _{|F}) \neq -2\) has lattice distance \(1\) from the origin.
\end{itemize}
\end{definition}
Let \(\Psi\) be a divisorial polytope. We may construct a complexity one polarized \(T\)-variety from the graded ring \(S\) given by: 
\[
S_k := \bigoplus_{u \in \Box \cap \frac{1}{k} M} H^0(\PP^1, \mathcal{O}(\lfloor k \cdot ( \Psi(u) +D) \rfloor),
\]
where \(D\) is some integral divisor of degree \(2\). We may also recover a divisorial polytope \(\Psi\) from any polarized complexity one \(T\)-variety \((X,L)\) such that \( H^0(X,L^k) = S_k\). Moreover Fano divisorial polytopes correspond to Fano \(T\)-varieties \((X,-K_X)\). Thus all Fano complexity one \(T\)-varieties may be described in this way. For more details of this construction and the correspondence see \cite{Sues2014} or \cite{Ilten2017}.

A Fano divisorial polytope fixes a distinguished representative of \(-K_X\) and specifies a linearisation of the action of \(T\) to \((X,-K_X)\). By the proof of \cite[Theorem 3.21]{Petersen2011} it is seen that the linearisation coincides with the canonical linearisation of \((X,-K_X)\), and so the moment map specified by the divisorial polytope, with image \(\Box\), coincides with the moment specified by the canonical linearisation.

We now recall some basic terminology for divisorial polytopes. The push-forward of the measure induced by \(\omega\) is known as the Duistermaat-Heckman measure, independent of the choice of \(\omega\) and which we denote by \(\nu\). Denote the standard measure on \(M_\RR\) by \(\eta\).

\begin{definition}
Let \(\Psi\) be a divisorial polytope.
\begin{itemize}
\item The degree of \(\Psi\) is the map \( \deg \Psi : \Box \to \RR\) given by \( u \mapsto \deg (\Psi(u))\). \\
\item The barycenter of \(\Psi\) is \(\bc(\Psi) \in \Box\), such that for all \(v \in N_\RR\):
\[
\langle \bc(\Psi), v \rangle = \int_\Box v \cdot \deg \Psi \ d \eta = \int_\Box v d \nu. 
\]
Note by the second equality we see \(\bc(\Psi) = \bc_\nu(\Box).\) \\
\item The volume of \(\Psi\) is defined to be:
\[
\vol \Psi = \int_\Box \deg \Psi \ d \eta = \int_\Box d \nu.
\]
\end{itemize}
\end{definition}

From here on let \((X,-K_X)\) be the polarized complexity one Fano \(T\)-manifold given by a fixed divisorial polytope \(\Psi: \Box \to \wdiv_\RR \PP^1\), with \(\bc_\nu(\Box) \neq 0\). We will calculate \(R(X)\) by considering first the product configurations and then the non-product ones.
\subsection{Product Configurations}
If \((\X,\L)\) is a product configuration then we have \(\X_0 \cong X\). Assuming \(X\) is non-toric, the maximality of \(T\) in \(\Aut(X)\) ensures that the restriction of \(\lambda\) to \(\X_0\) is a one parameter subgroup of \(T\), given by a choice of \(w \in N\).
We then have:
\begin{align*}
\DF_t(\X,\L)(w) &=   \DF(\X,\L)(w) + \frac{(1-t)}{V} \int_{X} (\max \theta_w - \theta_w) \omega^n \\ &= \langle  \bc(\Psi), w \rangle + \frac{(1-t)}{\vol \Psi} \int_\Box \max_{x \in \Box} \langle x,w \rangle  - \langle \cdot, w \rangle d \eta \\
&= t \langle  \bc(\Psi), w \rangle + (1-t)\max_{x \in \Box} \langle x,w \rangle.
\end{align*}
Let \(q \in N_\RR\) be the point of intersection of the ray generated by \(-\bc(\Psi)\) with \(\partial \Box\). Applying (1) we then obtain:
\[
\sup(t | \DF_t(\X,\L) \ge 0) = \frac{|q|}{|q-\bc_\nu(\Box)|}.
\]
\subsection{Non-Product Configurations}
We recall, from \cite{Ilten2017}, the description of special fibers of non-product special configurations. Let \(X\) be a Fano \(T\)-variety of complexity \(1\), corresponding to the Fano divisorial polytope \(\Psi\). The toric central fibers of non-product special test configurations were explicitely described in \cite{Ilten2017}. Namely there exists some \(y \in \PP^1\), with at most one of \(\Psi_z\) having non-integral slope at any \(u \in \Box\) for \(z \neq y\), such that \(\X_0\) is the toric variety corresponding to the following polytope:
\begin{equation*}
\Delta_y := \Big\{(u,r) \in M_\RR \times \RR \; \Big| \; u \in \Box,\; -1-\sum_{z \neq y} \Psi_z(u) \leq r \leq 1+\Psi_y(u)\Big\}.\label{eq:special-fibre}
\end{equation*}

The induced \(\CC^*\)-action on \(\X_0\) is given by the one-parameter subgroup of 
 \(T' = T \times \CC^*\) corresponding to \(v'=(-mv,m) \in N \times \ZZ\), for some \(v \in N\). In fact it turns out, from \cite{Ilten2017}, it is enough to consider those configurations with \(m=1\).
\begin{proposition}
Let \((\X,\L)\) be a special non-product test configuration with toric special fiber and induced \(\CC^*\)-action \(v' = (-v,1) \) as above. We then have:
\[
\DF_t( \X, \L )  = t \langle \bc(\Delta_y), v' \rangle + (1-t) \max_{x \in \Delta_y} \langle x, v' \rangle.
\]
\end{proposition}
\begin{proof}
In \cite{Ilten2017} the formula \(\DF(\X,\L) = \langle \bc(\Delta_y,v'\rangle\) is given. Note that the Hamiltonian function, by definition, satisfies \(\theta_w(x) = \langle \mu(x),w \rangle\). We may then calculate the remaining integrals on the image of the moment map, \(\Delta_y\).
\end{proof}
By \cite{Ilten2017} there are a finite number of possible special fibers \(\X_0\) of special test configurations of \((X,L)\). Label the corresponding polytopes \(\Delta_1,\dots,\Delta_m\). Set \(H := N_\RR \times \RR^+\).
\begin{proposition}
For any non-product configuration \((\X,\L)\) with special fiber one of the \(\Delta_i\) above, let \(\sigma_i\) be the cone of outer normals to \(\Delta_i\) at the unique point of intersection of \(\partial \Delta_i\) with the ray generated by \(-\bc(\Delta_i)\). Denote this point of intersection by \(q_i\). Then:
\[
\sup(t | \DF_t(\X,\L) \ge 0) =\begin{cases} 
     \frac{|q_i|}{|q_i - \bc(\Delta_i)|} & \sigma_i \cap H \neq \emptyset ;\\ \\
      \frac{|q|}{|q-\bc_\nu(\Box)|} & \sigma_i \cap H = \emptyset.
   \end{cases}
\]
\end{proposition}
\begin{proof}
Extend \(\DF_t(\X,\L)\) linearly to the whole of \(N_\RR \times \RR\). In the case \(\sigma_i \cap H \neq \emptyset\) we may apply (1) from Corollary 2 with \(P = \Delta_i\) and \(b = \bc(\Delta_i)\). Otherwise we may apply (2), noting that \(\pi(\Delta_i) = \Box\) and \(\pi(\bc(\Delta_i)) = \bc_\nu( \Box)\).
\end{proof}
\begin{proof}[of Theorem 1.1]
With Remark 2 in mind, observe that a special test configuration must either be product or non-product. Any non-product configurations \(\Delta_i\) with \(\sigma_i \cap H \neq \emptyset\) have their contribution to the infimum already accounted for and we may exclude them. The result follows.
\end{proof}
\begin{proof}[of Corollary 1.2]
We leave it to the reader to check that for each threefold we have \(n_i \notin H\) for every \(\Delta_i\) associated to a special test configuration. The divisorial polytopes and Duistermaat-Heckman measures may be found in \cite{Sues2014}. We may then calculate \(R(X)\) using just the base polytope \(\Box\) and its Duistermaat-Heckman barycenter.
\end{proof}
\begin{acknowledgements}\label{ackref}
I am very grateful to my supervisor Hendrik S{\"u}{\ss} for his advice and the many useful conversations we had whilst this paper was in preparation. I would also like to thank my anonymous referee, who provided many excellent comments and corrections.
\end{acknowledgements}

%
%
\end{document}